\documentclass[11pt]{article}

\usepackage[pagebackref,colorlinks=true,pdfpagemode=none,urlcolor=blue,
linkcolor=blue,citecolor=blue]{hyperref}

\usepackage{amsmath,amsfonts,amssymb,amsthm}
\usepackage{mathrsfs}

\usepackage{verbatim}
\usepackage{color}
\usepackage{url}
 \usepackage[all]{xy}
\usepackage{fancyhdr}

\usepackage[margin=1.75in]{geometry}

\newtheorem{thm}{Theorem}

\newtheorem{lem}{Lemma}

\newtheorem{definition}{Definition}

\newtheorem{remark}{Remark}

\newcommand{\sabs}[1]{\left|#1\right|}
\newcommand{\sparen}[1]{\left(#1\right)}
\newcommand{\norm}[1]{\sabs{\sabs{#1}}}

\numberwithin{equation}{section}
\title{ Stability Analysis for Magnetic Resonance Elastography\thanks{\footnotesize
This work was supported by ERC Advanced Grant Project MULTIMOD--267184.}}

\date{September 12, 2014}
\author{Habib Ammari\thanks{\footnotesize Department of
Mathematics and Applications, Ecole Normale Sup\'erieure, 45 Rue d'Ulm, 75005 Paris, France
(habib.ammari@ens.fr, alden.waters@ens.fr,  hai.zhang@ens.fr).} \and
Alden Waters\footnotemark[2] \and Hai Zhang\footnotemark[2]
}
\begin{document}
\maketitle

\begin{abstract}
We consider the inverse problem of finding unknown elastic parameters  from internal measurements of displacement fields for tissues. The measurements are made on the entirety of a smooth domain. Since tissues can be modeled as quasi-incompressible fluids, we examine the Stokes system and consider only the recovery of shear modulus distributions.  Our main result is to establish Lipschitz stable estimates on the shear modulus distributions from internal measurements of displacement fields.  These estimates imply convergence of a numerical scheme known as the Landweber iteration scheme for reconstructing the shear modulus distributions.
\end{abstract}

\section{Introduction}
We consider the problem of Magnetic Resonance Elastography (MRE). Using internal measurements of time-harmonic displacement fields offers the possibility of a high-resolved reconstruction of shear modulus distributions, in MRE; see \cite{zhou}. High resolution is important in the detection of cancerous anomalies in their early stages \cite{ammaribio}.

In this paper, we provide stability estimates for reconstructing the shear modulus
from internal measurements of displacement fields. For doing so, we first reduce the time-harmonic elasticity system to the Stokes system. Then we will follow the general approach in \cite{bgelliptic}. Our other main references are \cite{steinhauer, MS, solono, otmar}. See \cite{ammari_book_elas, garapon, zhou,gen1,gen2,seo,seo_book,gen3} for recent works on the inverse problem in MRE. For recent books and reviews on other inverse problems from internal measurements we refer to \cite{ammaribio, kuchment, nachman, seo_book}.

From a mathematical standpoint, inverse problems from internal measurements typically involve more measurements than parameters of interest and allow the recovery of the interesting parameters from an often redundant system of partial differential equations. Many times exact algebraic inversions may not be available. This paper addresses the specific problem of the Stokes system.  We reduce the problem of finding the shear modulus to the inversion of an over-determined system of partial differential equations. We
prove that under certain hypotheses, this system is elliptic and satisfies
the Lopatinskii condition.  Moreover, we show that in three dimensions additional internal measurements are needed. For other stability results on inverse problems from internal data for scalar equations we refer the reader to \cite{wenjia, laurent,  triki, honda, steinhauer, triki2}.

Ultimately the goal of examining Stokes system is their use in MRE to detect cancerous anomalies. We use the stability analysis of elliptic systems to prove the convergence of a numerical scheme known as the Landweber iteration scheme. This iteration scheme has already shown success in other simpler models \cite{laure}.

The paper is organized as follows. In section \ref{sec2} we derive the Stokes system from the elasticity equations as the compressional modulus goes to infinity. In section \ref{sec3} we provide some preliminary results on over-determined systems of partial differential equations.  Section \ref{sec-estimate-3d} is to prove the stability result of reconstructing the shear modulus from internal measurements of the displacement field in three dimensions.
In  section \ref{sec-estimate-2d}, the two dimensional case is considered.
In
section \ref{secfinal} we state the convergence result of the corresponding Landweber scheme.
The paper ends with some concluding remarks.

\section{Derivation of Stokes System from Elasticity Equations} \label{sec2}
Let $\Omega$ denote a simply-connected compact and smooth domain in $\mathbb{R}^d$ where $d=2, 3$ with $\mathcal{C}^\infty$-boundary $\partial\Omega$. We consider
\begin{align*}
f(x)=(f_1(x),f_2(x),\ldots,f_d(x)):\Omega\rightarrow \mathbb{R}^d.
\end{align*}
 We use the Einstein summation convention for the rest of this paper.

For two matrices $A$ and $B$, we let
 \begin{align*}
A:B=a_{ij}b_{ij}.
\end{align*}
We define the Hilbert spaces $H^{m}(\Omega)^d$ for $m\in \mathbb{N}$, as the completion of the space of $f(x)\in \mathcal{C}_c^{\infty}(\Omega)^d$ such that
\begin{align*}
\sum\limits_{i=1}^m \int\limits_{\Omega}\sparen{\nabla^i f(x):\nabla^i f(x)+|f(x)|^2}\,dx<\infty .
\end{align*}
We write $|\nabla f|^2=\nabla f:\nabla f$ from now on.
For any $u\in H^{1}(\Omega)^d$, we let
\begin{align*}
2\nabla^s u=\nabla u+(\nabla u)^t,
\end{align*}
where $\nabla u$ is the matrix $(\partial_ju_i)_{i,j=1}^d$ with $u_i$ as the $i$-th component of $u$, and the superscript $t$ denotes the transpose.
Let $\mu(x)\in \mathcal{C}^1(\Omega)$, then we define the conormal derivative
\begin{align*}
2\frac{\partial u}{\partial \nu}=\mu(x)\sparen{\nabla u+(\nabla u)^t}n,
\end{align*}
where $n$ is the outward unit normal to the boundary $\partial\Omega$.


We consider the following boundary value problem for the elasticity equations
\begin{align}\label{model}
\left\{
\begin{array}{lr}
\nabla(\lambda(x)\nabla\cdot u_{\lambda})+\omega^2 u_{\lambda}(x)+2\nabla\cdot\mu(x)\nabla^s u_{\lambda}(x)=0 \,\quad \mathrm{in}\,\, \Omega ,\\ \\
u_{\lambda}(x)=F(x) \,\quad \mathrm{on} \,\quad \partial\Omega \\
\end{array}
\right.
\end{align}
with $\mu(x), \lambda(x) \in \mathcal{C}^1(\bar{\Omega})$ the Lam\'{e} coefficients
(respectively, the shear and the compressional modulus) satisfying
\begin{align}
\label{assumpmin}
& \lambda\geq \lambda_{min}= \min\{\lambda(x): x\in \bar{\Omega}\} >0, \\
& \mu\geq \mu_{min}= \min\{\mu(x): x\in \bar{\Omega}\} >0.\label{assumpmin2}
\end{align}

The solution $u_{\lambda}(x)$ is such that
\begin{align*}
u_{\lambda}(x):\Omega\rightarrow \mathbb{R}^d .
\end{align*}
It is known that the solution $u_{\lambda}(x)$ exists and is unique. In particular,  $\nabla^su_{\lambda}(x)\in L^2(\Omega)^d$ if $F(x)\in H^{1/2}(\partial\Omega)$,
$\lambda, \mu \in L^\infty(\Omega)$ and satisfy (\ref{assumpmin}) and (\ref{assumpmin2})
and $\nabla^su_{\lambda}(x)\in H^4(\Omega)^d$ under the additional assumptions that $\mu(x), \lambda(x)\in \mathcal{C}^4(\bar{\Omega})$, $F\in H^{9/2}(\partial\Omega)^d$. We need the latter regularity assumption for later stability estimates in sections \ref{sec-estimate-3d} and \ref{sec-estimate-2d}.


The Poisson ratio $\sigma$ of the anomaly is given in terms of the Lam\'{e} coefficients by
\begin{align*}
\sigma=\frac{\lambda/\mu}{1+2\lambda/\mu}.
\end{align*}
It is known in soft tissues $\sigma\approx1/2$ or equivalently $\lambda>>\mu$. This makes it difficult to reconstruct both parameters $\mu$ and $\lambda$ simultaneously \cite{manduca},\cite{fatemi}. Therefore we first construct asymptotic solutions to the problem (\ref{model}) when $\lambda_{\min}\rightarrow \infty$. We loosely follows \cite{garapon} and \cite{alimit} which consider piecewise constant Lam\'{e} coefficients. We show that in the limit, the elasticity equations (\ref{model}) reduces to the following Stokes system
\begin{align}\label{stokesdata}
\left\{
\begin{array}{lr}
\omega^2 u(x)+2\nabla\cdot\mu(x)\nabla^s u(x)+\nabla p(x)=0 \,\quad \mathrm{in}\,\, \Omega ,\\ \\
\nabla\cdot u(x)=0 \,\quad \mathrm{in} \,\quad \Omega , \\\\
u(x)=F(x) \,\quad \mathrm{on} \,\quad \partial\Omega ,\\ \\
\int\limits_{\Omega}p(x)\,dx=0.
\end{array}
\right.
\end{align}


\begin{thm} \label{thm1}
Suppose that $\omega^2$ is not an eigenvalue of the problem (\ref{stokesdata}) with $F(x)=0$, then there exists a positive constant $C$ which is independent of $\lambda$ such that the following error estimates hold for $\lambda_{min}$ large enough
\begin{align} \label{error-estimate}
\norm{u_{\lambda}-u}_{H^1(\Omega)^d}\leq \frac{C}{\sqrt{\lambda_{min}}}.
\end{align}
\end{thm}
\begin{proof}
We recall the following identities $\forall v\in H^1(\Omega)^d$:
\begin{align}\label{divid}
\int\limits_{\Omega}(\nabla p)\cdot v\,dx= -\int\limits_{\Omega}(\nabla\cdot v ) p\,dx+\int\limits_{\partial\Omega}(v\cdot n)p\,dx
\end{align}
and
\begin{align}\label{muid}
\int\limits_{\Omega}\nabla\cdot (\mu \nabla^su)\cdot v \,dx =-2\int\limits_{\Omega}\mu\nabla^su:\nabla^sv \,dx+\int\limits_{\Omega}\frac{\partial u}{\partial\nu}\cdot v \,dx.
\end{align}
We will also need the following lemma.
\begin{lem}\label{korn}[Korn's inequality]
Let $\Omega$ be as above. Let $u\in H^1_0(\Omega)^d$ then
\begin{align*}
\int\limits_{\Omega}|\nabla u|^2\,dx \leq 2\int\limits_{\Omega}|\nabla^s u|^2\,dx,
\end{align*}
\end{lem}
\begin{proof}
c.f., for instance, \cite{ammari_book_elas}.
\end{proof}


We first prove that if $\omega^2$ is not an eigenvalue of (\ref{stokesdata}) then for $\lambda$ sufficiently large it is not an eigenvalue of the problem (\ref{model}). We start by assuming that $v$ is in fact an eigenvector for (\ref{model}) but not for (\ref{stokesdata}).  We integrate by parts using (\ref{divid}) and (\ref{muid}) to obtain the following identity
\begin{align*}
& \int\limits_{\Omega}\lambda|\nabla\cdot v|^2\,dx+4\int\limits_{\Omega}\mu|\nabla^s v|^2\,dx=\omega^2\int\limits_{\Omega}|v|^2\,dx.
\end{align*}
Without loss of generality, we assume that $v$ is normalized in $L^2(\Omega)^d$. Then
\begin{align*}
\int\limits_{\Omega}|\nabla\cdot v|^2\,dx\leq \frac{\omega^2}{\lambda_{min}} .
\end{align*}
Setting $q=\lambda\nabla\cdot v$ in (\ref{model}) , we see that the following holds
\begin{align}\label{error}
\left\{
\begin{array}{lr}
\omega^2 v(x)+2\nabla\cdot\mu(x)\nabla^s v(x)+\nabla q(x)=0 \,\quad \mathrm{in}\,\, \Omega , \\ \\
\nabla\cdot v=\mathcal{O}\sparen{\frac{1}{\sqrt{\lambda_{min}}}} \,\quad \mathrm{in} \,\, \Omega , \\ \\
v(x)=0 \,\quad \mathrm{on} \,\quad \partial\Omega ,\\
\end{array}
\right.
\end{align}
where the term $\mathcal{O}\sparen{\frac{1}{\sqrt{\lambda_{min}}}}$ is in the space $L^2(\Omega)^d$. Because $\omega^2$ is not an eigenvalue of (\ref{stokesdata}) by assumption, then there is a constant $C$  by Poincar\'{e} Inequality  and Lemma \ref{korn} such that
$$
\norm{v}_{L^2(\Omega)^d}\leq \frac{C}{\sqrt{\lambda_{min}}}.
$$
Recalling that $v$ is normalized in $L^2(\Omega)^d$, this is impossible for sufficiently large $\lambda_{min}$. Hence we can conclude that $\omega^2$ is not an eigenvalue of (\ref{model}). It follows that if the system (\ref{stokesdata}) has a unique solution then so does the system (\ref{model}).

For the second step, we prove (\ref{error-estimate}). We denote $u_{E}=u_{\lambda}-u$ where $u_{\lambda},u$ are solutions to the problems (\ref{model}) and (\ref{stokesdata}) respectively.  It is clear that $u_{E}=0$ on $\partial\Omega$.

We can derive the following energy identity by subtracting (\ref{stokesdata}) from (\ref{model}) and using (\ref{divid}) and (\ref{muid}):
\begin{align*}
&\int\limits_{\Omega}\lambda(x)|\nabla\cdot u_E|^2\,dx+4\int\limits_{\Omega}\mu(x)|\nabla^s u_E|^2\,dx- \omega^2\int\limits_{\Omega}|u_{E}|^2\,dx= \int\limits_{\Omega}p\nabla\cdot u_E .
\end{align*}
Let the bilinear form $B(\cdot,\cdot)$ on the space $H_0^1(\Omega)^d$ be denoted by
\begin{align*}
B(v,v)=\int\limits_{\Omega}\lambda(x)|\nabla\cdot v|^2+4\int\limits_{\Omega}\mu(x)|\nabla^s v|^2\,dx ,
\end{align*}
it then follows that
\begin{align} \label{inequality11}
B(u_E,u_E)-\omega^2\int\limits_{\Omega}|u_E|^2\,dx\leq \norm{\nabla\cdot u_E}_{L^2(\Omega)}\norm{p}_{L^2(\Omega)} .
\end{align}
Using Korn's inequality, we see that $B(\cdot,\cdot)$ is coercive. By the theory of collectively compact operators
(see \cite{compact}) it follows that if the main estimate (\ref{error-estimate}) holds for $\omega=0$ then it also holds for any $\omega^2$ which is not an eigenvalue of (\ref{stokesdata}). For this, we prove only for the case $\omega^2=0$. By the inequality (\ref{inequality11}),
\begin{align*}
B(u_{E},u_{E})\leq C\norm{\nabla\cdot u_{E}}_{L^2(\Omega)},
\end{align*}
where $C=\norm{p}_{L^2(\Omega)}$ is a bounded number independent of $\lambda$.
As a result we get
\begin{align*}
4\mu_{min}\norm{u_{E}}_{H^1(\Omega)^d}^2+\lambda_{min}\norm{\nabla\cdot u_E}^2_{L^2(\Omega)}\leq C \norm{\nabla\cdot u_E}_{L^2(\Omega)}.
\end{align*}
Thus we can derive that
\begin{align*}
\norm{\nabla\cdot u_E}_{L^2(\Omega)}\leq \frac{C}{\lambda_{min}},
\end{align*}
which further implies
\begin{align*}
\mu_{min}\norm{u_E}^2_{H^1(\Omega)^d}\leq \frac{C}{\lambda_{min}}.
\end{align*}
The main estimate (\ref{error-estimate}) follows immediately.
\end{proof}

\section{Preliminaries on Over-determined Elliptic Boundary-Value Problems} \label{sec3}
In this section, we present some basic properties about over-determined elliptic boundary-value
problems which plays a key role in our stability estimates in Section \ref{sec-estimate-3d} and \ref{sec-estimate-2d}.
The presentation follows closely to the ones in \cite{solono,otmar}. We present it here for the convenience of the reader.

We first recall the definition of ellipticity in the sense of Douglis-Nirenberg.
Consider the (possibly) redundant system of linear partial differential equations
\begin{align}\label{redundant}
&\mathcal{L}(x,\frac{\partial}{\partial x})y=\mathcal{S} ,\\ \nonumber
&\mathcal{B}(x,\frac{\partial}{\partial x})y=\phi
\end{align}
for $y$ unknown functions $y=(y_1,\ldots,y_m)$ comprising in total of $M$ equations. Here $\mathcal{L}(x,\frac{\partial}{\partial x})$ is a matrix differential operator of dimension $M\times d$ with entries $L_{ij}(x,\frac{\partial}{\partial x})$. For each $1\leq i\leq M$, $1\leq j\leq m$ and for each point $x$. The entry  $L_{ij}(x,\frac{\partial}{\partial x})$ is a polynomial in $\frac{\partial}{\partial x_i}$ $i=1,\ldots,d$. If the system is redundant, then there are possibly more equations than unknowns, $M\geq m$. The matrix $\mathcal{B}(x,\frac{\partial}{\partial x})$ has entries $B_{ij}(x,\frac{\partial}{\partial x})$ for $1\leq k\leq Q, 1\leq j\leq m$ consisting of $Q$ equations at the boundary. The operators are also polynomial in the partials of $x$. Naturally, the vector $\mathcal{S}$ is a vector of length $M$, and $\phi$ is a vector of length $Q$.

\begin{definition}\label{DNumbers}[c.f.\cite{DN},\cite{DL}]
Let integers $s_i,t_j\in \mathbb{Z}$ be given for each row $1\leq i\leq M$ and column $1\leq j\leq m$ with the following property: for $s_i+t_j\geq 0$ the order of $L_{ij}$ does not exceed $s_i+t_j$. For $s_i+t_j<0$, one has $L_{ij}=0$. Furthermore, the numbers are normalized so that for all $i$ one has $s_i\leq 0$. The numbers $s_i,t_j$ are known as Douglis-Nirenberg numbers.

The principal part of $\mathcal{L}$ for this choice of numbers $s_i, t_j$ is defined as the matrix operator $\mathcal{L}^0$ whose entries are composed of those terms in $L_{ij}$ which are exactly of order $s_i+ t_j$.

The principal part $\mathcal{B}^0$ of $\mathcal{B}$ is composed of the entries which are composed of those terms in $B_{kj}$ which are exactly of order $\sigma_k+t_j$. The numbers $\sigma_k$, $1\leq k\leq Q$ are computed as
\begin{align*}
\sigma_k=\max\limits_{1\leq j\leq m}(b_{kj}-t_j)
\end{align*}
with $b_{kj}$ denoting the order of $B_{kj}$. Real directions with $\xi\neq0$ and
$$
\mathrm{rank \,}\mathcal{L}^0(x,i\xi)<m
$$
are called characteristic directions of $\mathcal{L}$ at $x$. The operator $\mathcal{L}$ is said to be (possibly) over-determined elliptic in $\Omega$ if $\forall x\in \overline{\Omega}$ and for all real nonzero vectors $\xi$ one has
\begin{align*}
\mathrm{rank \,}\mathcal{L}^0(x,i\xi)=m.
\end{align*}

\end{definition}


We next recall the following Lopatinskii boundary condition.
\begin{definition}
Fix $x\in \partial\Omega$ and let $\nu$ be the inward unit normal vector at $x$. Let $\zeta$ be any non-zero tangential vector to $\Omega$ at $x$. We consider the line $\{x+z\nu, z>0\}$ in the upper half plane and the following system of ODE's
\begin{align}
&\mathcal{L}^0(x,i\zeta+\nu\frac{d}{dz})\tilde{y}(z)=0 \qquad z>0, \label{LS-1}\\
&\mathcal{B}^0(x,i\zeta+\nu\frac{d}{dz})\tilde{y}(z)=0  \qquad z=0.  \label{LS-2}
\end{align}
We define the vector space $V$ of all solutions to the system (\ref{LS-1})-(\ref{LS-2}) which are such that $\tilde{y}(z)\rightarrow 0$ as $z\rightarrow\infty$. If $V=\{0\}$, then we say that the Lopatinskii condition is fulfilled for the pair $(\mathcal{L},\mathcal{B})$ at $x$.
\end{definition}


Now, let $\mathcal{A}$ be the operator defined by
$$
\mathcal{A}=(\mathcal{L},\mathcal{B}).
$$
Then the equations (\ref{redundant}) read as $\mathcal{A}y=(\mathcal{S},\phi)$.

Let $\mathcal{A}$ acts on the space
\begin{align*}
D(p,l)=W_{p}^{l+t_1}(\Omega)\times \ldots \times W_p^{l+t_m}(\Omega)
\end{align*}
with $l\geq 0,\, p>1$. Here $W_p^{\alpha}$ denotes the standard Sobolev space with $\alpha$'s order partial derivatives in the $L^p$ space. With some regularity assumptions on the coefficients of $\mathcal{L}$ and $\mathcal{B}$,
$\mathcal{A}$ is bounded with range in the space
\begin{align*}
R(p,l)=W_p^{l-s_1}(\Omega)\times \ldots \times W_p^{l-s_m}(\Omega)\times W_p^{l-\sigma_1-\frac{1}{p}}\times \ldots \times W_p^{l-\sigma_q-\frac{1}{p}}(\partial\Omega).
\end{align*}

We have the following result, see \cite[Theorem 1]{otmar}.
\begin{thm}\label{invertible}
Let the integers $l\geq 0, p>1$ be given. Let $(\mathcal{S},\phi)\in\mathcal{R}(p,l)$. Let the Douglis-Nirenberg numbers $s_i$ and $t_j$ be given for $\mathcal{L}$ and $\sigma_k$ be as in Definition \ref{DNumbers}. Let $\Omega$ be a bounded domain with boundary in $\mathcal{C}^{l+\max t_j}$. Also assume that $p(l-s_i)>d$ and $p(l-\sigma_k)>d$ for all $i$ and $k$. Let the coefficients $L_{ij}$ be in $W_p^{l-s_i}(\Omega)$ and the coefficients of $B_{kj}$ be in $W^{l-\sigma_k-\frac{1}{p}}$. The following statements are equivalent
\begin{enumerate}
\item $\mathcal{L}$ is over-determined elliptic and the Lopatinskii covering condition is fulfilled for $(\mathcal{L},\mathcal{B})$ on $\partial\Omega$.
\item There exists a left regularizer $\mathcal{R}$ for the operator $\mathcal{A}=\mathcal{L}\times\mathcal{B}$ such that
$$
\mathcal{R}\mathcal{A}=\mathcal{I}-\mathcal{T}
$$
with $\mathcal{T}$ compact from $R(p,l)$ to $D(p,l)$.
\item The following a priori estimate holds
\begin{align*}
\sum\limits_{j=1}^m\norm{y_j}_{W_p^{l+t_j}(\Omega)}\leq C_1\sparen{\sum\limits_{i=1}^M\norm{\mathcal{S}_i}_{W_p^{l-s_i}(\Omega)}+\sum\limits_{k=1}^Q\norm{\phi_k}_{W_p^{l-\sigma_j-\frac{1}{p}}(\partial\Omega)}} \\ \qquad +C_2\sum\limits_{t_j>0}\norm{y_j}_{L^p(\Omega)},
\end{align*}
where $y_j$ is the $j$-th component of the solution of $y$.
\end{enumerate}
\end{thm}

\section{Main Stability Estimate in Dimension Three}\label{sec-estimate-3d}
We set $d=3$ for this section. We show that stability estimates for reconstructing $\mu$ are possible by using two sets of internal measurements of the displacement fields.

\begin{thm}\label{conditional}
Let $(u_1,p_1)$ and $(\tilde{u}_1,\tilde{p}_1)$ be solutions to (\ref{stokesdata}) with different boundary conditions.  In other words we set $F(x)=F_1(x)$ and $F(x)=\tilde{F}_1(x)$ in (\ref{stokesdata}) for the respective solutions but they share $\mu=\mu_1$. We assume that there exists a positive constant $C$  independent of $(x,\xi)\in T^*\bar{\Omega}$, where $T^* \bar{\Omega}$ denotes the cotangent space, such that
\begin{align}
|(\nabla^su_1(x)\xi)\times\xi|+|(\nabla^s\tilde{u}_1(x)\xi)\times\xi|\geq C|\xi|^2.
\label{stabilizerscond}
\end{align}
Let $(u_2,p_2)$ and $(\tilde{u}_2,\tilde{p}_2)$ be solutions to the Stokes system (\ref{stokesdata}) with $\mu=\mu_2$ and $F(x)=F_1(x)$ and $F(x)=\tilde{F}_1(x)$, respectively. Assume that $\mu_1, \mu_2 \in \mathcal{C}^4(\bar{\Omega})$ and $\mu_1=\mu_2$ on $\partial\Omega$.  Then
there exist a constant $C$, depending on $\|\mu_2\|_{\mathcal{C}^{l+2}(\bar{\Omega})}$, and a finite dimensional subspace $K$, of $W^4_2(\Omega)$ such that
\begin{align}\label{mainest}
\norm{\mu_1-\mu_2}_{W^{4}_2(\Omega)}\leq C \sparen{\norm{u_1-u_2}_{W^{5}_2(\Omega)}+\norm{\tilde{u}_1-\tilde{u}_2}_{W^{5}_2(\Omega)}},
\end{align}
provided that $(\mu_1-\mu_2) \perp K $.
\end{thm}

\begin{remark}
It is possible that it requires more than two sets of measurements, but the arguments here work the same. In general we need a finite number of internal measurements, but we do not prove their existence here. Notice that this condition is very different from the usual assumption which is typically something of the form $|\nabla u|\geq c>0$ \cite{seo_book} for the conductivity equation or $\det\nabla^s u\neq 0$ for elastic wave equations \cite{otmar}.
Here, $\det$ denotes the determinant. Naturally this provides the necessary rank condition in Definition \ref{DNumbers}.
\end{remark}

We start by eliminating the pressure terms from the Stokes systems. We consider the equations for $i=1,2$,
\begin{align*}
\nabla\cdot\mu_i\nabla^su_i+\omega^2u_i+\nabla p_i=0.
\end{align*}
Taking the cross product of both sides yields
\begin{align*}
\nabla\times\nabla\cdot\mu_i\nabla^su_i+\omega^2\nabla\times u_i=0.
\end{align*}
Setting $\mu=\mu_1-\mu_2$ and $w=u_1-u_2$, if we subtract the first equation from the second equation we obtain
\begin{align}\label{mu}
\nabla\times\nabla\cdot\mu \nabla^s u_1=g,
\end{align}
where
\begin{align*}
g=-\nabla\times[\nabla\cdot\mu_2\nabla^s w]-\omega^2\nabla\times w.
\end{align*}
It is then clear that there is a constant $C$ depending on $\|\mu_2\|_{\mathcal{C}^{2+l}(\bar{\Omega})}$ such that
\begin{align}\label{HS}
\norm{g}_{W^{l}_p(\Omega)}\leq C\norm{w}_{W^{l+3}_p(\Omega)}
\end{align}
for all $l\geq 0, p >1$. In order to determine $\mu$, we view the identity (\ref{mu}) as a system of over-determined second-order partial differential equations for the unknown function $\mu$. Indeed,
one can recast $\nabla\times[\nabla\cdot\mu\nabla^su_1]$ in the format of a linear operator $L_{u_1}(x,\frac{\partial}{\partial x})$ which is a polynomial of degree 2 in $\frac{\partial}{\partial x_i}$, $i=1,2,3$, acting on $\mu$.

We want to apply Theorem \ref{invertible} to the over-determined system of $\mu$. We encounter the difficulty
that $L_{u_1}$ is not elliptic. Actually, the principal symbol of the linear operator $L_{u_1}$ can be calculated which turned out to be
\begin{align*}
\mathcal{L}^0(x,i\xi)=(\nabla^su_1(x)\xi)\times \xi,
\end{align*}
and which is clearly not elliptic. Therefore, we have to augment the operator $L_{u_1}$ with a second set of measurements.

We set
\begin{align*}
\mathcal{L}_*=(L_{u_1},L_{\tilde{u}_1}) \qquad \mathcal{B}_*=(\mbox{Trace on } \partial \Omega,\mbox{Trace on } \partial \Omega),
\end{align*}
and consider the new augmented system:
\begin{align}\label{system}
\mathcal{L}_*[\mu]=(g,\tilde{g}), \quad \mathcal{B}_*[\mu]=(0,0).
\end{align}

It is clear that the condition (\ref{stabilizerscond}) ensures that the new augmented system is elliptic.
We now check that the Lopatinskii condition is satisfied.
Define matrix $A$ as $\nabla^s u_1$ and $\tilde{A}=\nabla^s\tilde{u}_1$.
The equations (\ref{LS-1}) and (\ref{LS-2}) for the Lopatinskii condition read as follows
\begin{align*}
&(Av\times v)\frac{d^2}{dz^2}\tilde{\mu}+i(A\zeta\times v + Av\times \zeta)\frac{d\tilde{\mu}}{dz}-(A\zeta\times\zeta)\tilde{\mu}=0 ,\\
&(\tilde{A}v\times v)\frac{d^2}{dz^2}\tilde{\mu}+i(\tilde{A}\zeta\times v + \tilde{A}v\times \zeta)\frac{d\tilde{\mu}}{dz}-(\tilde{A}\zeta\times\zeta)\tilde{\mu}=0,\\
&\tilde{\mu}(0)=0,
\end{align*}
where $\zeta\cdot \nu=0$.
Now we can apply the transpose of $((Av\times v),(\tilde{A}v\times v))$  to the first two equations in the system.
This results in an ODE of the form
\begin{align}\label{ode}
a\frac{d^2}{dz^2}\tilde{\mu}+ib\frac{d\tilde{\mu}}{dz}+c\tilde{\mu}=0
\end{align}
with $a$ $b$, $c$ being real numbers depending on the entries of $A,\tilde{A}, \zeta, v$.  The expression for $a$ is given by
\begin{align*}
a=|Av\times v|^2+|\tilde{A}v\times v|^2>0,
\end{align*}
which by our assumptions is positive. The solutions are linear combinations of the fundamental solutions $\exp(\lambda_1z)$ and $\exp(\lambda_2 z)$, where $\lambda_1,\lambda_2$ are given by
\begin{align*}
\lambda_{1, 2}=\frac{-bi\pm\sqrt{-b^2-4ac}}{2a}.
\end{align*}
Clearly there is at most one exponentially decaying fundamental solution, from which we can derive that the Lopitanskii covering condition is satisfied.

\begin{proof}[Proof of Theorem \ref{conditional}]
We have check that the augmented system $(\mathcal{L}_*, \mathcal{B}_*)$ is elliptic and the Lopatinskii condition is satisfied.  To apply Theorem \ref{invertible}, we set $\phi=0$, $d=3$, $t_j=2 (1\leq j \leq 6)$, $s_i=0 (i=1)$, $l=2$ and $p=2$. Note that on the perpendicular space to the kernel K, the operator $\mathcal{A}_*=(\mathcal{L}_*, \mathcal{B}_*)$ is invertible, thus we have the estimate
\begin{align*}
\norm{\mu}_{W^{4}_2(\Omega)}\leq C (\norm{g}_{W^{2}_2(\Omega)} + \norm{\tilde{g}}_{W^{2}_2(\Omega)})  \leq C' (\norm{w}_{W^{5}_2(\Omega)}+\norm{\tilde{w}}_{W^{5}_2(\Omega)}).
\end{align*}
The conclusion that K is finite dimensional follows from standard compactness argument. This completes the proof of the Theorem.
\end{proof}

%
%


\section{Main Estimate in Dimension Two} \label{sec-estimate-2d}
We set $d=2$. We show in this section that stability estimates are possible without additional set of internal measurements.


\begin{thm}\label{conditional2}
Let $(u_1,p_1)$, and $(u_2,p_2)$ be solutions to the equation (\ref{stokesdata}) with coefficients $\mu_1$ and $\mu_2$ respectively.  Assume that $\mu_1, \mu_2 \in \mathcal{C}^4(\bar{\Omega})$ and $\mu_1=\mu_2$ on $\partial\Omega$ and the following non-degeneracy condition holds:
\begin{equation} \label{cond-2d}
 \det(\nabla^s u_1(x)) \neq 0, \quad x\in \bar{\Omega}.
\end{equation}
Then there exists a non-zero constant $C$, depending on $\|\mu_2\|_{\mathcal{C}^4(\bar{\Omega})}$, and a finite dimensional subspace $K$, of $W^4_2(\Omega)$ such that
\begin{align}\label{mainest2}
\norm{\mu_1-\mu_2}_{W^{4}_2(\Omega)}\leq C\norm{ u_1-u_2}_{W^{5}_2(\Omega)}
\end{align}
provided $(\mu_1-\mu_2) \perp K$.
\end{thm}

\begin{remark}
The non-degeneracy condition (\ref{cond-2d}) is equivalent to the following one
 $$
 \partial_{x_{1}}u_1^1(x) \neq 0
 $$
where $u^1_1$ denotes the first entry in the solution $u_1$. This follows from the equation that $\nabla \cdot u_1=0$.
\end{remark}

\begin{proof}
We apply the similar arguments as in the previous section. However instead of the operator $\nabla\times$ which is not defined for $d=2$, we use the analogue which is the operator $(\partial_{x_1},-\partial_{x_2})\cdot$  to eliminate the pressure term. (This comes from deRahm's theorem, and there is an analogue of this operator in any dimension, but we concentrate on the cases $d=2,3$ which are relevant for the physical models).

Set $\mu=\mu_1-\mu_2$, we can derive that
\begin{align} \label{eq1}
(\partial_{x_1},-\partial_{x_2})\cdot (2\nabla\cdot \mu \nabla^s u_1)=g
\end{align}
where
$$
\norm{g}_{W^{l}_p(\Omega)}\leq C \norm{u_1-u_2}_{W^{l+2}_p(\Omega)}
$$
for all $l \geq 0, p>1$. If we calculate the principal symbol of the linear operator (acting on the function $\mu$) on the left hand side of (\ref{eq1}),
we obtain
\begin{align*}
L^0(x,i\xi)=2|\xi|^2\partial_{x_{1}}u_1^1(x),
\end{align*}
where we have used the fact that $\nabla \cdot u_1=0$.  Thus the operator is elliptic by our assumption. We now check for the Lopatinskii boundary condition. The associated equations are
\begin{align*}
& \frac{d^2}{dz^2}\tilde{\mu}= 2\partial_{x_1}u^1_1|\zeta|^2\tilde{\mu}, \\
& \tilde{\mu}(0) =0.
\end{align*}
The solutions can be written as $\tilde{\mu}=C_1 \exp(\lambda_1z)+C_2 \exp(\lambda_2z)$ where $\lambda_1$ and $\lambda_2$ are roots of the following characteristic equation
\begin{align*}
\lambda^2=2\partial_{x_1}u^1_1(x)|\zeta|^2.
\end{align*}
One can show that the only solution which decays to zero at the infinity is the trivial solution. Thus the Lopatinskii condition is satisfied. The rest of the proof follows by the same arguments as in the proof of Theorem \ref{conditional}.
\end{proof}



\section{Imaging Shear Modulus Distributions} \label{secfinal}

Let $u_{m}$ be the measured displacement field, which corresponds to the true shear modulus distribution $\mu_{{\rm tr}}$. In order to reconstruct $\mu$ from $u_{m}$, we introduce the following  discrepancy functional  between the computed and measured displacement fields:
\begin{align*}
\mathcal{J}[\mu]=\frac{1}{2}\int\limits_{\Omega}\sabs{u-u_{m}}^2\,dx,
\end{align*}
where $u$ is the solution to (\ref{stokesdata}), and minimize $\mathcal{J}[\mu]$ over all admissible $\mu$. For a given $\mu$, we let $v$ denote the solution to the following adjoint system~\cite{ammari_book_elas}:
\begin{align*}
\left\{
\begin{array}{lr}
2\nabla\cdot\mu\nabla^s v+\omega^2 v +\nabla p =
\overline{u-u_m} \quad \mathrm{in} \quad \Omega ,\\ \\
\nabla\cdot v =0 \quad \mathrm{in} \quad \Omega ,\\ \\
v=0 \quad \mathrm{on}\quad  \partial\Omega ,\\ \\
\int\limits_{\Omega}p \,dx=0 .
\end{array}
\right.
\end{align*}
The Fr\'echet derivative $D\mathcal{J}[\mu]$ of $\mathcal{J}$ is given by
\cite{ammari_book_elas}
\begin{align*}
<D\mathcal{J}[\mu],\delta\mu>=\int\limits_{\Omega}\delta\mu\nabla^sv : \nabla^su\,dx.
\end{align*}
As such, we can identify  $D\mathcal{J}[\mu]$ with
$\nabla^sv:\nabla^su .$

Using a gradient descent method, we can then numerically minimize $\mathcal{J}$. If the initial guess is $\mu_0$ we can then update it with the following scheme
\begin{align}\label{bestguess}
\mu_{n+1}(x)=\mu_n(x)-\sigma D\mathcal{J}[\mu_n](x) \qquad x\in \Omega,\,\, n\geq 0
\end{align}
with $\sigma$ being the step size; see again \cite{ammari_book_elas}.

On the other hand, if we let $\mathcal{F}$ denote the following map
\begin{align*}
\mathcal{F}: \mu\mapsto u,
\end{align*}
then, according to \cite{laure}, we have
\begin{align*}
D\mathcal{J}[\mu]=(D\mathcal{F}[\mu])^*(\mathcal{F}[\mu]-\mathcal{F}[\mu_{\rm tr}]),
\end{align*}
where the superscript $*$ denotes the adjoint.

The result is that the optimal control scheme (\ref{bestguess}) can be identified with the
Landweber iteration scheme given by
\begin{align*}
\mu_{n+1}(x)=\mu_n(x)-\sigma(D\mathcal{F}[\mu])^*(\mathcal{F}[\mu]-\mathcal{F}[\mu_{\rm tr}])(x) \qquad x\in \Omega,\,\, n\geq 0 .
\end{align*}

From \cite[Appendix A]{laure}, the following convergence result in $H^4(\Omega)$ for the Landweber (or equivalently the optimal control) scheme holds.
\begin{thm}
Let $d=2$.
Assume that the assumptions of Theorem \ref{conditional2} are satisfied and $K$ is trivial.  If, for  sufficiently small $\epsilon_0$,
\begin{align*}
\norm{\mu_0 -\mu_{\rm tr}}_{H^4(\Omega)}<\epsilon_0 ,
\end{align*}
then
\begin{align*}
\norm{\mu_n-\mu_{\rm tr}}_{H^4(\Omega)} \rightarrow 0 \qquad \mathrm{as} \, \, \,  n\rightarrow \infty.
\end{align*}
\end{thm}

In three dimensions, we should modify the discrepancy function $\mathcal{J}$ as follows
$$
\mathcal{J}[\mu]=\frac{1}{2} \bigg[ \int\limits_{\Omega}\sabs{u-u_{m}}^2\,dx  +  \int\limits_{\Omega}\sabs{\tilde{u}-\tilde{u}_{m}}^2\,dx\bigg],
$$
where $u_m$ and $\tilde{u}_m$ corresponds to two different boundary conditions $F=g$ and $F=\tilde{g}$ and are such that condition (\ref{stabilizerscond}) holds. Accordingly, $\mathcal{F}$ should be changed to
$$
\mathcal{F}: \mu \mapsto \left( \begin{array}{l} u \\ \tilde{u} \end{array} \right),
$$
where $u$ and $\tilde{u}$ are respectively the solutions to the Stokes system with boundary conditions $g$ and $\tilde{g}$.
The following convergence result holds true in three dimensions. Its proof is exactly the same as in the two-dimensional case.  It is worth emphasizing that the theorem requires adding two energy functionals for the iteration to converge.
\begin{thm}
Let $d=3$.
Assume that the assumptions of Theorem \ref{conditional} are satisfied and $K$ is trivial.
If, for sufficiently small $\epsilon_0$,
\begin{align*}
\norm{\mu_0-\mu_*}_{H^4(\Omega)}<\epsilon_0,
\end{align*}
then we have
\begin{align*}
\norm{\mu_n-\mu_{\rm tr}}_{H^4(\Omega)} \rightarrow 0 \qquad \mathrm{as} \, \, \,  n\rightarrow \infty.
\end{align*}
\end{thm}

It is worth noticing that a good initial guess for the reconstruction problem of $\mu$ in both the two and three dimensional cases was recently found in \cite{zhou}.

\section{Concluding Remarks}

In this paper we have derived Lipschitz stability estimates for the reconstruction of shear modulus distributions from internal measurements of displacement fields. Our estimates yield to a convergence result for the Landweber iteration scheme. It would be very interesting to make use of multifrequency measurements in order to remove the kernel $K$ of the over-determined system $(\mathcal{L},\mathcal{B})$. Another challenging problem is to extend the present analysis to anisotropic shear modulus distributions. These important problems will be the subject of a future work.

\renewcommand\refname{\large References}
\bibliographystyle{abbrv}
\bibliography{elastography}


\end{document}